\documentclass{article}
\usepackage{amssymb}
\usepackage{latexsym}
\usepackage{amsmath}
\usepackage{mathrsfs}
\usepackage[all]{xy}
\usepackage{enumerate}
\usepackage{amsthm}
\usepackage[dvips]{graphicx}
\usepackage{psfrag}
\usepackage{ifthen}

\newtheorem{theorem}{Theorem}[section]
\newtheorem{lemma}[theorem]{Lemma}
\newtheorem{metalemma}[theorem]{Meta Lemma}
\newtheorem{proposition}[theorem]{Proposition}
\newtheorem{corollary}[theorem]{Corollary}
\theoremstyle{definition}
\newtheorem{definition}[theorem]{Definition}

\newtheorem{example}[theorem]{Example}
\newtheorem{convention}[theorem]{Convention}

\newenvironment{thm}{\begin{theorem}}{
\end{theorem}}

\newenvironment{prop}{\begin{proposition}}{
\end{proposition}}

\newenvironment{cor}{\begin{corollary}}{
\end{corollary}}

\newenvironment{lem}{\begin{lemma}}{
\end{lemma}}

\newenvironment{defn}{\begin{definition}}{
\end{definition}}

\newcommand{\noi}{\noindent}
\newcommand{\tr}{\textrm}

\newcommand{\insertimage}[2]
	   {\includegraphics[width=#2\textwidth]{#1}}

\newcommand{\mo}{{-1}}
\newcommand{\pmo}{{\pm 1}}
\newcommand{\gt}{\mapsto}
\renewcommand{\tilde}{\widetilde}
 
\newcommand{\FRS}{\ensuremath{F_{R(S)}}}
\newcommand{\Z}{\mathbb{Z}}
\newcommand{\bra}{\langle}
\newcommand{\kett}{\rangle}
\newcommand{\ol}{\overline}
\newcommand{\rk}{\tr{rank}}

\newcommand{\ncl}{\tr{ncl}}

\newcommand{\F}{\widetilde{F}}

\newcommand{\x}{\xi}

\newcommand{\stab}{\tr{stab}}

\newcommand{\brakett}[1]{{\bra #1 \kett}}
\newcommand{\showcomments}{yes}
\newcommand{\comment}[1]
	   {\ifthenelse{\equal{\showcomments}{yes}}
	     {\footnotemark\marginpar{\sffamily{\tiny
		   \addtocounter{footnote}{-1}\footnotemark#1
}\normalfont}}{}}

\newcommand{\Hom}{\tr{Hom}_F}

\newcommand{\bel}[1]{\begin{equation}\label{#1}}
\newcommand{\be}{\begin{equation}}
\newcommand{\ee}{\end{equation}}

\newcommand{\bxyk}{\brakett{x,y}}

\newcommand{\stw}{\stab(w)}
\newcommand{\bgwk}{\brakett{\gamma_w}}
\newcommand{\ahat}{\widehat{\alpha}}
\newcommand{\Diag}{\tr{Diag}_F}
\renewcommand{\x}{\ol{x}}
\newcommand{\y}{\ol{y}}

\title{The equation $w(x,y)=u$ over free groups: an algebraic approach} 
\author{Nicholas W.M. Touikan\\Department of Mathematics and Statistics, McGill University\\Montr\'eal, Qu\'ebec, Canada\\\textit{touikan@math.mcgill.ca}} 

\begin{document} 
\maketitle
\abstract{Using the theory developed by Olga Kharlampovich, Alexei
  Miasnikov, and, independently, by Zlil Sela to describe the set of
  homomorphisms of a f.g. group $G$ into a free group $F$, we describe
  the solutions to equations with coefficients from $F$ and unknowns
  $x,y$ of the form $w(x,y)=u$, where $u$ lies in $F$ and $w(x,y)$ is
  a word in $\{x,y\}^{\pm 1}$. We also give an example of a single
  equation whose solutions cannot be described with only one ``level''
  of automorphisms.}
\section{Introduction}
Solving systems of equations over free groups has been a very
important topic in group theory. A major achievement was the algorithm
due to Makanin and Razborov \cite{Makanin-1982, Razborov-1987} which
produces a complete description of the solution set of an arbitrary
finite system of equations over a free group. In practice, however, the
algorithm is quite complicated and does not readily imply the results
of this paper.

Much has already been said about solutions to certain types of systems
of equations. Solutions of systems of equations in one unknown over a
free group were described in 1960 by Lyndon \cite{Lyndon-1960}.  In
1971, Hmelevski{\u\i} gave in \cite{Hmelevskii-eqn} an algorithm to
decide solvability as well as a description of the solutions of
equations in unknowns $x,y$ with coefficients in a free group $F$ of
the form $w(x,y)=u$, and $t(x,F)=u(y,F)$. In 1972 Wicks
\cite{Wicks-1972} also described a method for find all the the
solutions of the equation $w(x,y)=u$. In his paper, Wicks gives a way
to find a finite set of solutions to an equation and shows how to
generate all the possible solutions from this finite set using
automorphisms. It has also been shown by Laura Ciabanu in
\cite{Ciobanu-2007} that there is a polynomial time algorithm to
determine if $w(x,y)=u$ has a solution. So far all the approaches have
been combinatorial.

In this paper we tackle the equation $w(x,y)=u$ from a different point
of view. We will use the theory developed by Olga Kharlampovich,
Alexei Miasnikov, and, independently, by Zlil Sela to describe the set
of homomorphisms of a f.g. group $G$ into a free group $F$. We start
by considering the \emph{fully residually} $F$ groups (also called the
\emph{Limit groups relative to} $F$) corresponding to the equation
$w(x,y)=u$. These groups were shown by Remeslennikov in
\cite{Remeslennikov-1989} to be key in the study of systems of
equations. We then systematically describe the possible
\emph{canonical $F$-automorphisms} of these groups and give the
possible \emph{Hom} (also called \emph{Makanin-Razborov}) diagrams
that arise.

In so doing we get an algebraic proof that solutions to equations of
the form $w(x,y)=u$, can be parametrized by a finite set of
\emph{minimal solutions} and a group of \emph{canonical
automorphisms}, which gives us a very explicit description of the
arising algebraic varieties (see Theorem \ref{thm:main-result}). We
also exhibit an equation $E(F,x,y)=1$ whose solutions cannot be
described this way (see Theorem \ref{thm:two-levels}). In particular,
we recover some of the aforementioned results of Hmelevski{\u\i} and
Wicks, but our description of the solutions is by far the most
transparent. In our opinion this paper also serves as an illustration
of some of the very important ideas and techniques that have recently
been applied fully residually free (or limit) groups.

\subsection{$F$-groups and Algebraic Geometry}
A complete account of the material in this section can be found in
\cite{BMR-1998}.  Fix a free group $F$. An equation in variables $x,y$
over $F$ is an expression of the form \[E(x,y)=1\] where $E(x,y) =
f_1{z_1}^{m_1}\ldots {z_n}^{m_n}f_{n+1}; ~f_i \in F, z_j \in \{x,y\}$
and $m_{k} \in \Z$. By an equation of the form $w(x,y)=u$ we mean an
equation \[{z_1}^{m_1} \ldots {z_n}^{m_n}u^\mo=1\] where $u \in F, z_j
\in \{x,y\}$.

We view an equation as an element of the group $F[x,y]=F*F(x,y)$. A
\emph{solution} of an equation is a substitution
\begin{equation}\label{eqn:solution}x\gt g_1, y\gt g_2;~g_i \in
F\end{equation} so that in $F$ the product $E(g_1,g_2)=_F 1$. A
\emph{system of equations} in variables $x,y$; $S(x,y)=S$; is a subset
of $F[x,y]$ and a solution of $S(x,y)$ is a substitution as in
(\ref{eqn:solution}) so that all the elements of $S(x,y)$ vanish in
$F$.

\begin{defn}
  A group $G$ equipped with a distinguished monomorphism \[
  i:F\hookrightarrow G\] is called an \emph{$F$-group} we denote this
  $(G,i)$. Given $F$-groups $(G_1,i_1)$ and $(G_2,i_2)$, we define an
  $F-$homomorphism to be a homomorphism of groups $f$ such that the
  following diagram commutes:\[\xymatrix{ G_1 \ar[r]^f & G_2 \\ F
    \ar[u]^{i_1} \ar[ur]_{i_2}&}\] We denote by $\Hom(G_1,G_2)$ the
  set of $F$-homomorphisms from $(G_1,i_1)$ to $(G_2,i_2)$.
\end{defn}
 
In the remainder the distinguished monomorphisms will in general be
obvious and not explicitly mentioned. It is clear that every mapping
of the form (\ref{eqn:solution}) induces an $F$-homomorphism
$\phi(g_1,g_2):F[x,y]\rightarrow F$, it is also clear that every $f
\in \Hom(F[x,y],F)$ is induced from such a mapping. It follows that we
have a natural bijective correspondence\[\Hom(F[x,y],F)
\leftrightarrow F\times F=\{(g_1,g_2)|g_i \in F\}\]

\begin{defn}
Let $S=S(x,y)$ be a system of equations. The subset \[V(S)=\{(g_1,g_2)
  \in F\times F| x\gt g_1, y\gt g_2 \tr{~is a solution of~}S\}\]
  is called the \emph{algebraic variety} of $S$.
\end{defn}
We have a natural bijective correspondence \[\Hom(F[x,y]/ncl(S),F)
\leftrightarrow V(S)\]

\begin{defn} The \emph{radical} of $S$ is the normal subgroup
  \[Rad(S)=\bigcap_{f\in \Hom(F[x,y]/ncl(S),F)}\ker(f)\] and we denote the
   \emph{coordinate group} of $S$ \[\FRS=F[x,y]/Rad(S)\]
\end{defn}
It follows that there is a natural bijective correspondence
\[\Hom(F[x,y]/ncl(S),F)\leftrightarrow \Hom(\FRS,F)\] so that
$V(S)=V(Rad(S))$. We say that $V(S)$ or $S$ is \emph{reducible} if it
is a union \[V(S)=V(S_1)\cup V(S_2); V(S_1)\subsetneq \cup V(S)
\supsetneq V(S_2)\] of algebraic varieties. An $F$-group $G$ is said
to be \emph{fully residually} F if for every finite subset $P\subset
G$ there is some $f_P \in \Hom(G,F)$ such that the restriction of
$f_P$ to $P$ is injective.

\begin{thm}[\cite{BMR-1998}]\label{thm:irred}
$S$ is irreducible if and only if $\FRS$ is fully residually $F$.
\end{thm}

\begin{thm}[\cite{BMR-1998}]\label{thm:factor}
Either $\FRS$ is fully residually $F$ or \[V(S)=V(S_1)\cup\ldots
V(S_n)\] where the $V(S_i)$ are irreducible and there are canonical
epimorphisms $\pi_i:\FRS\rightarrow F_{R(S_i)}$ such that each $f \in
\Hom(\FRS,F)$ factors through some $\pi_i$.
\end{thm}

\begin{cor}\label{cor:obvious}
  If $F[x,y]/ncl(S)$ is fully residually $F$ then $\FRS =
  F[x,y]/ncl(S)$.
\end{cor}

\subsection{Rational Equivalence}
\begin{defn} An \emph{$F$-automorphism} of $F[x,y]$ is an automorphism
  \[\phi:F[x,y]\rightarrow F[x,y]\] such that the restriction $\phi|_F$
  is the identity. Two systems of equations $S,T$ are said to be
  \emph{rationally equivalent} if $\phi(S)=T$, for some $\phi \in
  Aut_F(F[x,y])$.
\end{defn}

\begin{prop}\label{prop:rat-equiv}\begin{itemize} 
  \item [(i)]$Aut_F(F[x,y])$ is generated by the elementary Nielsen
    transformations on the basis $\{F,x,y\}$ that fix $F$ elementwise.
    \item [(ii)] If $S,T$ are rationally equivalent via $\phi \in
      Aut_F(F[x,y])$, then the natural map $\tilde{\phi}$ in the
      commutative diagram below is an isomorphism.
  \end{itemize}
  \[\xymatrix{F[x,y] \ar[r]^\phi \ar[d]^\pi & F[x,y]\ar[d]^\pi\\
    \FRS \ar[r]^{\tilde{\phi}} & F_{R(T)}}\]
\end{prop}

\begin{prop}\label{prop:w-primitive} Suppose $w(x,y)$ is a
  primitive (by \emph{primitive} we mean an element that belongs to
  some basis) element of $F(x,y)$, then there exist words $X(u,z),
  Y(u,z)$ such that the set of solutions of $w(x,y)=u$ correspond to
  the set of pairs \[(x,y)=(X(u,z),Y(u,z))\] where $z$ takes arbitrary
  values in $F$.
\end{prop}
\begin{proof}
Let $S=\{w(x,y)u\}$. By assumption there is $\phi \in Aut_F(F[x,y])$ that
sends $w(x,y)$ to $x$ and $\phi$ extends to an $F$ automorphism of
$F[x,y]$. This means that $S$ is rationally equivalent to
$T=\{xu^\mo\}$. The first thing to note is that $F_{R(T)}$ is a free
group, hence so is $\FRS$. $\Hom(F_{R(T)},F)$ is given by \[
V(T) = \{(x,y)\in F\times F| x=u, y \in F\}\] the result now follows by
precomposing with $\tilde{\phi}^\mo$, as defined in Proposition
\ref{prop:rat-equiv}.
\end{proof}

\begin{lem}\label{cor:QH-fun} Suppose the free group $F(x,y)$ on
  generators $\{x,y\}$ admits a presentation
  \[F(x,y)=\brakett{\xi,\zeta,p|[\xi,\zeta]p^\mo}\] where $\xi, \zeta,
  p \in F(x,y)$. Then the mapping
  $\phi(\xi)=x,\phi(\zeta)=y,\phi(p)=[x,y]$, extends to an
  automorphism $\phi:F(x,y)\rightarrow F(x,y)$.
\end{lem}
\begin{proof} Notice that the basis elements $x,y$ of $[x,y]$
  obviously satisfy the identity $[x,y][x,y]^\mo=1$, so the mapping
  $\phi$ gives an automorphism.
\end{proof}

\subsection{Splittings}
We assume the reader is familiar with Bass-Serre theory, so we only
describe enough to explain our notation.

\begin{defn} A \emph{graph of groups} $\mathcal{G}(A)$ consists of a
  connected directed graph $A$ with vertex set $VA$ and edges $EA$.
  $A$ is directed in the sense that to each $e\in EA$ there are
  functions $i:EA\rightarrow VA, t:EA\rightarrow VA$ corresponding to
  the \emph{initial and terminal} vertices of edges. To $A$ we
  associate the following:\begin{itemize}
    \item To each $v \in VA$ we assign a \emph{vertex group} $G_v$.
    \item To each $e \in EA$ we assign an \emph{edge group} $G_e$.
    \item For each edge $e \in EA$ we have monomorphisms
      \[\sigma_e:G_e\rightarrow G_{i(e)}, ~\tau_e:G_e\rightarrow
      G_{t(e)}\] we call the maps $\sigma_e,\tau_e$ \emph{boundary
    monomorphisms} and the images of these maps \emph{boundary subgroups}.
  \end{itemize}
\end{defn}

A graph of groups has a fundamental group denoted
$\pi_1(\mathcal{G}(A))$. We say that a
group \emph{splits} as the fundamental group as a graph of groups if
$G=\pi_1(\mathcal{G}(A))$ and refer to the data $D=(G,\mathcal{G}(A))$
as a \emph{splitting}.

\begin{defn}[Moves on $\mathcal{G}(A)$]\label{defn:moves} 
We have the following moves on $\mathcal{G}(A)$ that do not change the
  fundamental group.
\begin{itemize}
\item \emph{Change the orientation of edges} in $\mathcal{G}(A)$, and
  relabel the boundary monomorphisms.
\item \emph{Conjugate boundary monomorphisms}, i.e. replace $\sigma_e$ by
  $\gamma_g\circ \sigma_e$ where $\gamma_g$ denotes conjugation by $g$
  and $g \in G_{i(e)}$.
\item \emph{Slide}, i.e. if there are edges $e,f$ such that
  $\sigma_e(G_e)=\sigma_f(G_f)$ then we change $X$ by setting $i(f)=t(e)$
  and replacing $\sigma_f$ by $\tau_e\circ \sigma_e^{-1} \circ \sigma_f$.
\item \emph{Folding}, i.e. if $\sigma_e(G_e)\leq A \leq G_{i(e)}$, then
  replace $G_{t(e)}$ by $G_{t(e)}*_{\tau_{e}(G_e)}A$, replace $G_e$ by a
  copy of $A$ and change the boundary monomorphism accordingly.
\item \emph{Collapse an edge $e$}, i.e. for some edge $e \in EA$, take
  the subgraph $star(e)=\{i(e),e,t(e)\}$ and consider the quotient of
  the graph $A$, subject to the relation $\sim$ that collapses
  $star(e)$ to a point. The resulting graph $A'=A/\sim$ is again a
  directed graph. Denote the equivalence class $v'=[star(e)] \in A'$,
  then we have $G_{v'}=G_{i(e)}*_{G_e}G_{t(e)}$ or $G_{i(e)}*_{G_e}$
  depending whether $i(e)=t(e)$ or not. For each edge $f$ of $A$
  incident to either $i(e)$ or $t(e)$, we have boundary monomorphisms
  $G_f \rightarrow G_{v'}$ given by $\sigma_f'=i\circ\sigma_f$ or
  $\tau_f'=i\circ\tau_f$, where $i$ is the one of the inclusion
  $G_{t(e)}\subset G_{v'}$ or $G_{i(e)}\subset G_{v'}$.
\item \emph{Conjugation}, i.e. for some $g \in G$ replace all the
  vertex groups by $G_{v}^{g}$ and postcompose boundary monomorphisms
  with $\gamma_g$ (which denotes conjugation by $g$).
\end{itemize}
\end{defn}

\subsection{The cyclic JSJ decomposition}

\begin{defn}\label{defn:dehntwist} An \emph{elementary cyclic splitting} $D$
  of $G$ is a splitting of $G$ as either a free product with
  amalgamation or an HNN extension over a cyclic subgroup. We define
  the \emph{Dehn twist along} $D$, $\delta_D$, as
  follows.\begin{itemize}
    \item If $G=A*_\brakett{\gamma}B$ then
    \[\delta_{D}(x)=\left\{ \begin{array}{ll}
     x & \tr{if}~x \in A\\ x^{\gamma}&\tr{if}~x \in B
      \end{array}\right.\]
    \item If $G= \brakett{A,t|t^\mo \gamma t = \beta}, \gamma, \beta
    \in A$ then \[\delta_{D}(x)=\left\{
\begin{array}{ll} x & \tr{if}~x \in A\\ t\beta &\tr{if}~x=t
      \end{array}\right.\]
\end{itemize}
A Dehn twist generates a cyclic subgroup of $Aut(G)$. A splitting such
that all the edge groups are nontrivial and cyclic is called a
\emph{cyclic splitting}.
\end{defn}

We can generalize the notion of a Dehn twist to arbitrary cyclic
splittings.

\begin{defn}\label{defn:gendehntwist} let $D$ be a cyclic splitting of $G$
  with underlying graph $A$ and let $e$ be an edge of of $A$. Then a
  \emph{Dehn twist} along $e$ is an automorphism that can be obtained
  by collapsing all the other edges in $A$ to get a splitting $D'$ of
  $G$ with only the edge $e$ and applying one the applicable
  automorphisms of Definition \ref{defn:dehntwist}
\end{defn}

\begin{defn}\begin{itemize}\item[(i)] A subgroup $H\leq G$ is elliptic
  in a splitting $D$ if $H$ is conjugable into a vertex group of $D$,
  otherwise we say it is hyperbolic.

\item[(ii)]Let $D$ and $D'$ be two elementary cyclic splittings of a
  group $G$ with boundary subgroups $C$ and $C'$, respectively. We say
  that $D'$ is \emph{elliptic in} $D$ if $C'$ is elliptic in $D$.
  Otherwise $D'$ is \emph{hyperbolic} in $D$
\end{itemize}
\end{defn}

A splitting $D$ of an $F$-group is said to be \emph{modulo} $F$ if the
subgroup $F$ is contained in a vertex group.

The following is proved in \cite{R-S-JSJ}:
\begin{thm}\label{thm:hyp-hyp}\begin{itemize}
  \item[(i)] Let $G$ be freely indecomposable (modulo $F$) and let
    $D',D$ be two elementary cyclic splittings of $G$ (modulo $F$).
    $D'$ is elliptic in $D$ if and only if $D$ is elliptic in $D'$.
  \item[(ii)] Moreover if $D'$ is hyperbolic in $D$ then $G$ admits a
    splitting $E$ such that one of its vertex groups is the
    fundamental group $Q=\pi_1(S)$ of a punctured surface $S$ such
    that the boundary subgroups of $Q$ are puncture subgroups.
    Moreover the cyclic subgroups $\brakett{d}, \brakett{d'}$
    corresponding to $D,D'$ respectively are both conjugate into $Q$.
\end{itemize}
\end{thm}

\begin{defn} A subgroup $Q\leq G$ is a quadratically hanging (QH)
  subgroup if for some cyclic splitting $D$ of $G$, $Q$ is a vertex
  group that arises as in item (ii) of Theorem \ref{thm:hyp-hyp}.
\end{defn}

Not every surface with punctures can yield a QH subgroup. By Theorem 3
of \cite{KM-IrredI}, the projective plane with puncture(s) and the
Klein bottle with puncture(s) cannot give QH subgroups. It has also
been noted that surfaces that can give QH subgroups must admit pseudo-Anosov
homeomorphisms.
 
\begin{defn} \begin{itemize} \item[(i)] A QH subgroup $Q$ of $G$ is a
    \emph{maximal QH} (MQH) subgroup if for any other QH subgroup $Q'$
    of $G$, if $Q \leq Q'$ then $Q=Q'$.
  \item[(ii)] Let $D$ be a splitting of $G$ with $Q$ be a QH vertex
    subgroup and let $C$ be a splitting of $Q$ with boundary subgroup
    $\brakett{c}$ then there is a splitting $D'$ of $G$ called a
    \emph{refinement of $D$ along C} such that $D$ is obtained from a
    collapse of $D'$ along an edge whose corresponding group is
    $\brakett{c}$.
\end{itemize}
\end{defn}

\begin{defn}\label{defn:almost-reduced} \begin{itemize}
  \item[(i)] A splitting $D$ is \emph{almost
      reduced} if vertices of valency one and two properly contain the
    images of edge groups, except vertices between two MQH subgroups
    that may coincide with one of the edge groups. 
  \item[(ii)] A splitting $D$ of $G$ is \emph{unfolded} if $D$ can not
      be obtained from another splitting $D'$ via a folding move (See
      Definition \ref{defn:moves}).\end{itemize}
\end{defn}

\begin{thm}[Proposition 2.15 of \cite{KM-JSJ}]\label{thm:jsj}  Let $H$
  be a freely indecomposable modulo $F$ f.g. fully residually $F$
  group. Then there exists an almost reduced unfolded cyclic splitting
  $D$ called the \emph{cyclic JSJ splitting of $H$ modulo $F$} with
  the following properties:\begin{enumerate}
    \item[(1)] Every MQH subgroup of $H$ can be conjugated into a vertex
    group in $D$; every QH subgroup of $H$ can be conjugated into one
    of the MQH subgroups of $H$; non-MQH [vertex] subgroups in $D$ are
    of two types: maximal abelian and non-abelian [rigid], every
    non-MQH vertex group in $D$ is elliptic in every cyclic splitting
    of $H$ modulo $F$.
    \item[(2)] If an elementary cyclic splitting $H=A*_CB$ or $H=A*_C$ is
    hyperbolic in another elementary cyclic splitting, then $C$ can be
    conjugated into some MQH subgroup.
    \item[(3)] Every elementary cyclic splitting $H=A*_CB$ or $H=A*_C$
    modulo $F$ which is elliptic with respect to any other elementary
    cyclic splitting modulo $F$ of $H$ can be obtained from $D$ by a
    sequence of moves given in Definition \ref{defn:moves}.
    \item[(4)] If $D_1$ is another cyclic splitting of $H$ modulo $F$ that
    has properties (1)-(2) then $D_1$ can be obtained from $D$ by a
    sequence of slidings, conjugations, and modifying boundary
    monomorphisms by conjugation (see Definition \ref{defn:moves}.)
  \end{enumerate}
\end{thm}

\begin{defn}\label{defn:canonical-automorphisms}
  (For simplicity we consider only the case where $\FRS$ is freely
  indecomposable modulo $F$.) Given $D$, a cyclic JSJ decomposition of
  $\FRS$ modulo $F$, we define the group $\Delta$ of \emph{canonical
  $F-$automorphisms} of $\FRS$ to be generated by the following:
\begin{itemize}
\item Dehn twists along edges of $D$; or by Dehn twists along edges $e'$
  obtained from refinements of $D$ along cyclic splittings of MQH
  subgroups; that fix $F\leq \FRS$.
\item Automorphisms of the abelian vertex groups that fix edge groups.
\end{itemize}
\end{defn}

The following Theorem is proved in \cite{KM-IrredII,Sela-DiophI}.

\begin{thm} If $\FRS \neq F$ and is freely indecomposable (modulo $F$)
  then it admits a non trivial cyclic JSJ decomposition modulo $F$.
\end{thm}

\subsection{The Structure of $\Hom(\FRS,F)$}

\begin{defn} \emph{A $Hom$ diagram} for $\Hom(G,F)$, denoted $\Diag(G,F)$,
  consists of a finite directed rooted tree $T$ such that the
  root, $v_0$, has no incoming edges and otherwise every vertex has at
  most one incoming edge along with the following data:
  \begin{itemize}
    \item To each vertex, except the root, $v$ of $T$ we
    associate a fully residually $F$ group $F_{R(S_v)}$.
  \item The group associated to each leaf of $T$ is a free product
    $F*H_1*\ldots*H_n$, where the $H_i$ are isomorphic to subgroups of
    $F$. (The $H_i$ can be thought as free variables)
    \item To each edge $e$ with initial vertex $v_i$ and terminal
    vertex $v_t$ we have a proper $F$-epimorphism $\pi_e:F_{R(S_{v_1})}
    \rightarrow F_{R(S_{v_2})}$
    \end{itemize}
  \end{defn} 

  We point out that in the work of Sela, the $Hom$ diagram is called a
  Makanin-Razborov diagram (relative to F) and that our fully
  residually $F$ groups are \emph{limit groups} (relative to F). The
  following theorem gives a finite parametrization of the solutions of
  systems of equations over a free group.

\begin{thm}[\cite{KM-IrredII,Sela-DiophI}]\label{thm:parametrization}
  For any system of equations $S(x_1,\ldots,x_n)$ there exists a $Hom$
  diagram $\Diag(\FRS,F)$ such that for every $f \in \Hom(\FRS,F)$
  there is a path \[v_0,e_1,v_1,e_2,\ldots ,e_{m+1},v_{m+1}\] from the
  root $v_0$ to a leaf $v_{m+1}$ such that\[f = \rho \circ
  \pi_{v_{m+1}} \circ \sigma_{v_{m}} \circ \ldots \circ \sigma_{v_1}
  \circ \pi_{e_1}\] where the $\sigma_{v_j}$ are \emph{canonical}
  $F$-automorphisms of $F_{R(S_{v_j})}$, the $\pi_j$ are epimorphisms
  $\pi_j:F_{R(S_{v_j})} \rightarrow F_{R(S_{v_{j+1}})}$ inside
  $\Diag(\FRS,F)$, and $\rho$ is any $F$-homomorphism
  $\rho:F_{R(S_{v_{m+1}})}\rightarrow F$ from the free group
  $F_{R(S_{v_{m+1}})}$ to $F$.
\end{thm}

\section{The system of equations $S=\{w(x,y)u^{-1}\}$}
\begin{defn}Let $\phi$ be a solution of $S$, then the \emph{rank} of
  $\phi$ is the rank of the subgroup $\brakett{\phi(x),\phi(y)}\leq
  F$.
\end{defn}
If all solutions of $S$ are of rank 1, then $V(S)$ is easy to describe
and is given in Section \ref{sec:easy}. If $S$ has solutions of rank
2, then there will be infinitely many such solutions. For this case we
will prove that $\Diag(\FRS,F)$ correspond to one
of three cases (see Figure \ref{homdiagrams}.)
\begin{figure}
\centering
\psfrag{F}{$F$}
\psfrag{F1}{$\FRS$}
\psfrag{F2}{$\FRS$}
\psfrag{F3}{$\FRS$}
\psfrag{Fs}{$F_1$}
\psfrag{p}{$V-V(S_1)$}
\psfrag{p1}{$\pi_1$}
\psfrag{p2}{$\pi_2$}
\psfrag{p3}{$\pi_3$}
\insertimage{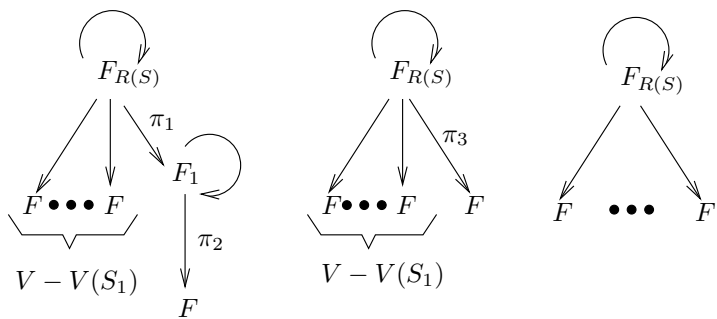}{0.75}
\caption{Hom diagrams corresponding to cases $1., 2.$, and $3.$ of
  Corollary \ref{cor:allowable-jsj}, $\pi_1,\pi_2,\pi_3$ are given in
  Proposition \ref{prop:rank1-solutions}.}
\label{homdiagrams}
\end{figure}
We will moreover describe the possible splittings of $\FRS$ and the
associated canonical automorphisms. This description along with
Theorem \ref{thm:parametrization}, will enable us to describe $V(S)$
as a set of pairs of words in $F$ (see Theorem \ref{thm:main-result}).

\subsection{Easy Cases and Reductions}\label{sec:easy}
By Proposition \ref{prop:w-primitive} we need only concern ourselves
with the case where $w(x,y)$ is not primitive. We state some results that
enable us to simplify matters:

\begin{lem}\label{lem:u-nontrivial} The equation $w(x,y)=1$ doesn't
  admit any rank 2 solutions.
\end{lem}

Let $\sigma_x(w)$ and $\sigma_y(w)$ be the exponents sums of $x$ and
$y$ respectively in the word $w(x,y)$. Then it is easy to see that
\bel{eqn:trivial-u}V(S)=\{(r^{n_1},r^{n_2})\in F\times F |r \in F;~
n_1\sigma_x(w)+n_2\sigma_y(w)=0\}\ee In this case we have that
$\FRS\approx F*<t>$ and the mapping $F[x,y]/\ncl(S)\rightarrow \FRS$
is given by the mapping \[\left\{\begin{array}{l} f \mapsto f, f \in
F\\ x \mapsto t^{r_x} \\ y \mapsto t^{r_y} \end{array}\right.\] where
$(r_x,r_y)$ is a generator of the subgroup $\{(a,b)\in \Z\oplus\Z |
a\sigma_x(w) + b\sigma_y(w) = 0\}.$

\begin{lem}\label{lem:w-not-properpower} If $w(x,y)=v(x,y)^n, n>1$
  then either the variety $V(\{w(x,y)u^\mo\})$ is empty or $u=r^n$ for
  some $r \in F$ and we have the equality $V(\{w(x,y)u^\mo\}) =
  V(\{v(x,y)r^\mo\})$.
\end{lem}

We will always assume that $w(x,y)$ is not a proper power.  Although
this may seem somewhat contrived, our reason for doing so is twofold:
firstly, requiring that an element is primitive is not enough; in our
theorems we want to exclude the case where $w(x,y)$ is a proper power
of a primitive element as, again, solutions are easy to describe.
Secondly, if $w(x,y)=v(x,y)^n$ with $n$ maximal, then in the cyclic
JSJ splitting of $\FRS$ modulo $F$, the edge group will be generated
by $v(x,y)$ and not $w(x,y)$. For the next result we need the
following theorem:

\begin{thm}[Main Theorem of \cite{Baumslag-1965}] 
\label{thm:Baumslag1965}
Let $w=w(x_1,x_2,\ldots,x_n)$ be an element of a free group $F$ freely
generated by $x_1,x_2,\ldots,x_n$ which is neither a proper power nor
a primitive. If $g_2,g_2,\ldots,g_n, g$ are elements of a free group
connected by the relation\[ w(g_1,g_2,\ldots,g_n)=g^m ~~(m>1)\] then
the rank of the group generated by $g_1,g_2,\ldots,g_n$ is at most
$n-1$.
\end{thm}

\begin{cor}\label{cor:u-not-properpower} Suppose that $w(x,y)$ is
  neither primitive nor a proper power. If $u=r^n, n>1$ is a proper
  power then the equation $w(x,y)=u$ doesn't have any rank 2
  solutions.
\end{cor}
\begin{proof}
  Suppose not then there is a solution $\phi:\FRS \rightarrow F$ such
  that $\x=\phi{x}, \y=\phi{y}$ and $[\x,\y]\neq 1$ which means that
  $\brakett{\x,\y}$ is free group of rank two. But we have the
  identity $w(\x,\y)=r^n$, which by Theorem \ref{thm:Baumslag1965}
  implies that rank of $\brakett{\x,\y}$ is at most one
  --contradiction.
\end{proof}

\subsection{Possible cyclic JSJ splittings of $\FRS$ and canonical
  automorphisms} 

\begin{lem}\label{lem:nss} Suppose that $w(x,y)$ is neither
  primitive nor a proper power. If $w(x,y)=u$ has a rank 2 solution
  then the group
  \[F[x,y]/\ncl(S) \approx F*_{u=w(x,y)}\brakett{x,y}\] is fully
  residually $F$ and, in particular, we have that \[\FRS =
  F*_{u=w(x,y)}\brakett{x,y}\]
\end{lem}
\begin{proof}
  Let $(\x,\y)$ be a rank 2 solution. Let $F_1 = \brakett{F,t|t^\mo u
    t = u}$, $F_1$ is a rank one free extension of a centralizer of
  $F$, and therefore is fully residually $F$. By definition
  $F-$subgroups are also fully residually $F$. Let
  $H=\brakett{\x,\y}\leq F$ and let $H'=t^\mo H t$. By Britton's Lemma we
  see that $H' \cap F = \brakett{u}$ and that \[\brakett{F,H}\approx
  F_{u=w({\x}^t,{\y}^t)}H'\approx F*_{u=w(x,y)}\brakett{x,y}\] so
  this gives an $F-$embedding $F*_{u=w(x,y)}\brakett{x,y}
  \hookrightarrow F_1$ so $F*_{u=w(x,y)}\brakett{x,y}$ is fully
  residually $F$.  By Corollary \ref{cor:obvious} we obtain the
  equality \[\FRS = F[x,y]/\ncl(S)\]
\end{proof}

\begin{lem}If $w(x,y)$ is not primitive nor a proper power then $\FRS
  = F*_{u=w(x,y)}\brakett{x,y}$ is freely indecomposable modulo $F$.
\end{lem}
\begin{proof}
  Suppose not. Since $\brakett{x,y}$ is a free group of rank 2, if it
  splits freely with nontrivial factors, then it must split as a free
  product of two cyclic groups. Since any splitting of $\FRS$ modulo
  $F$ must also be modulo $w(x,y)$ we have that $w(x,y)$ must lie in
  one of these free cyclic factors, contradicting the hypotheses of
  the lemma.
\end{proof}

Given this first decomposition as an amalgam, we wish to see how it
can be refined to a cyclic JSJ decomposition modulo $F$. By the
Freiheitssatz, the subgroup $\bxyk \leq \FRS$ is free of rank 2. So to
investigate cyclic JSJ decomposition modulo $F$, we must first look at
the possible cyclic splitting of $\bxyk$. Our main tool will be the
following theorem of Swarup:

\begin{thm}[Theorem 1 of \cite{Swarup-1986}]\label{thm:Swarup} (A) Let
  $G=G_1*_HG_2$ be an amalgamated free product decomposition of a free
  group $G$ with $H$ finitely generated. Then, there is a non-trivial
  free factor $H'$ of $H$ such that $H'$ is a free factor of either
  $G_1$ or $G_2$.

(B) Let $G=J*_{H,t}$ be an HNN decomposition of a free group $G$ with
  $H$ finitely generated. Then there are decompositions $H=H_1*H_2,
  J=J_1*J_2$ with $H_1$ non trivial such that $H_1$ is a free factor
  of $J_1$ and $t^\mo H_1 t$ is conjugate in $J$ to a subgroup of
  $J_2$.
\end{thm}

\begin{cor}\label{cor:rank-formula} If  $G=G_1*_{\brakett{\gamma}}G_2$
  is an amalgamated free product decomposition of a free group over a
  nontrivial cyclic subgroup, then $\rk(G)=\rk(G_1)+\rk(G_2)-1$.
\end{cor}

\begin{lem}\label{lem:allowable-splittings}Let $G$ be
  a free group of rank 2 and let $w\in G$ be non primitive, and not a
  proper power. Then the only possible almost reduced (see Definition
  \ref{defn:almost-reduced}) nontrivial cyclic splittings of $G$ as
  the fundamental group of a graph of groups with $w$ elliptic are as
  \begin{itemize} \item[(i)] a star of groups, specifically a graph of
    groups whose underlying graph is simply connected, consisting of a
    center vertex $v_c$ and a collection of peripheral vertices
    $v_1,\ldots,v_m$ connected to $v_c$ by an edge. The group
    associated to $v_c$, called the \emph{central group}, is free of
    rank 2 and each edge group is nontrivial, cyclic and is a proper
    finite index subgroup of the associated ``peripheral'' vertex
    group; or \item[(ii)]as an HNN extension
  \[G=\brakett{H,t|t^\mo p t = q}; p,q \in H-\{1\} \] where $w \in H$
  and $H$ is another free group of rank 2. Moreover we have that
  $H=\brakett{p}*\brakett{q}$ i.e.  $G=\brakett{p,t}$.\end{itemize}
\end{lem}

\begin{proof}
  Let $D$ be a splitting of $G$. If $G$ splits as a free product with
  amalgamation $G=G_1*_\brakett{\gamma}G_2$ then if $\gamma$ is not
  trivial, Corollary \ref{cor:rank-formula} forces one of the factors
  to be cyclic. Since we are assuming almost reducedness we must have
  that the edge group is a finite index subgroup of one of the cyclic
  factors. Suppose $G_2$ is a cyclic factor and let $z$ be a generator
  of $G_2$. Then the free group $G$ is obtained by adjoining the
  $n^{th}$ root $z$ of the element $\gamma \in G_1$, which is a free
  group of rank 2. It is however impossible to have a further
  splitting $G_1*{\brakett{\gamma}}*G_2*_{\brakett{\gamma'}}*G_3$ with
  $G_2$ and $G_3$ cyclic and with $\brakett{\gamma},\brakett{\gamma'}$
  proper finite index subgroups of $G_2,G_3$ (resp.) since then, by an
  easy computation using normal forms, it would be possible to get a
  counter example to commutation transitivity, which must hold in a
  free group. The general star case follows.

  If the underlying graph of $D$ is simply connected and one of the
  edge groups is trivial, then we can collapse $D$ to a free product
  $G_1*G_2$ with nontrivial factors, and with $w$ lying in one of the
  vertex groups, by Grushko's Theorem we must have
  $\rk(G_1)=\rk(G_2)=1$ and our assumption that $w$ is elliptic in $D$
  and not a proper power forces $w$ to be primitive --contradiction.
  We have therefore covered the case where the underlying graph is
  simply connected.

  If the underlying graphs has two cycles (and a nontrivial vertex
  group), then we would have a proper epimorphism $G\rightarrow
  F(a,b)$ which contradicts the Hopf property.

  \noi \emph{Claim: If $G=\brakett{H,t|t^\mo p t = q}$, then $H$ is a
    free group of rank 2.} By Theorem \ref{thm:Swarup} $(B)$ and
  conjugating boundary monomorphisms we can arrange so that
  \bel{eqn:split1}H=H_1*H_2 \tr{~with~} p\in H_1 \tr{~and~} q \in
  H_2\ee Theorem \ref{thm:Swarup} $(B)$ moreover gives us that without
  loss of generality we can assume that $\brakett{q}$ is a free factor
  of $H_2$. This means that \bel{eqn:split2}H_2=H_2'*\brakett{q}\ee
  Letting $H'=H_1*H_2'$ we get that $H=H'*\brakett{q}$ so combining
  (\ref{eqn:split1}) and (\ref{eqn:split2}) gives us a presentation
  $G=\brakett{H',t,q|t^\mo p t = q}$ which via a Tietze transformation
  gives us \bel{eqn:split3}G=\brakett{H',t | \emptyset}\ee which
  forces $H'$ to be cyclic which means that $H$ has rank 2. Moreover,
  we see immediately that $H=\brakett{p}*\brakett{q}$.
\end{proof}

We denote by $\Delta$ the group of canonical $F-$automorphisms of
$\FRS$ (see Definition \ref{defn:canonical-automorphisms}.) 

\begin{convention}
  Whenever we have a ``star'' splitting of the subgroup
  $\brakett{x,y}$, as given in item (i) in the statement of Lemma
  \ref{lem:allowable-splittings}, we will collapse the whole splitting
  to a single vertex group. The first reason being that the Dehn
  twists around the edge groups fixing the central group act
  trivially. Secondly, by uniqueness of $n^{th}$ roots in a free
  group, we see that any mapping of the central group into a free
  group has at most a unique extension to the whole group. It follows
  that to describe solutions to the equation, the collapsed splitting
  is sufficient.
\end{convention}

\begin{cor}\label{cor:allowable-jsj}
There are three possible classes of cyclic JSJ decomposition modulo
$F$ of $\FRS$:\begin{enumerate}
\item $\FRS\approx F*_{u=w(x,y)}\brakett{x,y}$ and $\Delta = \bgwk$,
  where $\gamma_w$ is the automorphism that
  extend the mapping:\[\gamma_w: \left\{\begin{array}{ll} f \mapsto
      f;& f \in F\\ z \mapsto w^\mo z w;& z \in \bxyk
    \end{array}\right.\]
\item The subgroup $\brakett{x,y}$ splits as a cyclic HNN-extension:
  \[\brakett{x,y}=\brakett{H,t|t^\mo p t = q}\] with $w(x,y)\in H$ so
  that $\FRS\approx F*_{u=w(x,y)}\brakett{H,t|t^\mo p t = q}$ and
  $\Delta = \brakett{\gamma_w,\tau}$ where these are the automorphisms
  that extend the mappings:\[\gamma_w: \left\{\begin{array}{ll} f
      \mapsto f;& f \in F\\ z \mapsto w^\mo z w;& z \in \bxyk
  \end{array}\right. ; \tau:\left\{\begin{array}{ll} z \mapsto z;& z\in
  \brakett{F,H}\\ t \mapsto tq&\end{array}\right.\]
\item $\FRS \approx F*_{u=w(x,y)}Q$ where $Q$ is a QH subgroup and, up
  to rational equivalence, $Q=\brakett{x,y,w|[x,y]w^\mo}$. $\Delta$
  is generated by the automorphisms extending the
  mappings:\[\gamma_w; ~ \delta_x:\left\{\begin{array}{l}x\mapsto yx\\
      \tr{identity
  on~} F\cup\{y\}\end{array}\right. ;
  ~\delta_y:\left\{\begin{array}{l}y\mapsto xy \\ \tr{identity on~}
  F\cup\{x\}\end{array}\right.\]
\end{enumerate}
\end{cor}
\begin{proof} Suppose first that the cyclic JSJ decomposition of
  $\FRS$ modulo $F$ has a QH subgroup $Q$. Then $Q$ must be a subgroup
  of $\bxyk$, in particular there must be a splitting of $\bxyk$
  modulo $w$ such that $Q$ is one of its vertex groups. By Lemma
  \ref{lem:allowable-splittings} we must either have that $Q=\bxyk$,
  or $\bxyk$ is an HNN extension of $Q$. Either way we must have that
  $Q$ is a free group of rank 2. The possible punctured surfaces $S$
  such that $\pi_1(S)$ is a free group of rank 2 are the once
  punctured torus or the once punctured Klein bottle, the latter is
  not allowed (see Theorem 3 of \cite{KM-IrredI}.) Moreover, we see
  that if $\bxyk$ is an HNN extension of $Q$ then the associated
  subgroups must be conjugate in $Q$, which would imply that $\bxyk$
  contains an abelian free group of rank 2 --contradiction. It follows
  from Corollary \ref{cor:QH-fun} that, up to rational equivalence,
  the only possibility is as in case 3. of the statement.

  The rest of the statement follows immediately from Lemma
  \ref{lem:allowable-splittings} and Definition
  \ref{defn:canonical-automorphisms}.
\end{proof}

\subsection{Solutions of rank 1}\label{sec:vs1}
We now consider solutions of rank 1. Although everything can easily be
described in terms of linear algebra, it is instructive to explain
this in terms of Hom diagrams and canonical automorphisms, because as
we shall see these provide examples of canonical epimorphisms that are
not \emph{strict} (see \cite{Sela-DiophI} for a definition.)

As we saw earlier, rank 1 solutions occur when we are solving
$w(x,y)=1$. More generally a rank 1 solutions occurs if and only if
$w(x,y) = u = v^d$ where $d=gcd(\sigma_x(w),\sigma_y(w))$;
$\sigma_x(w), \sigma_y(w)$ denote the exponent sums of $x, y$ in
$w(x,y)$. Corollary \ref{cor:u-not-properpower} states that if $d>1$,
but $w(x,y)$ not primitive and not a proper power, then all solution
of $w(x,y)=u$ have rank 1. If $d=1$ then $w(x,y)=u$ may have both rank
1 and rank 2 solutions.

Let $S_1=\{w(x,y)u^\mo,[x,y]\}$, then all rank 1 solutions must factor
through $F_{R(S_1)}$. If $d > 1$ then, since all solutions are rank 1,
we must have we in fact have $Rad(\{w(x,y)u^\mo\}) =
ncl(\{w(x,y)u^\mo, [x,y]\})$. As a set, these solutions are easy to
describe:

\bel{eqn:vs1}V(S_1)=\{(u^{n_1},u^{n_2})\in F\times F |n_1 \sigma_x(w)
+ n_2\sigma_y(w) = d\}\ee Let $p,q$ be integers such that
\bel{eqn:gcd}p\sigma_x(w)+q\sigma_y(w)=d\ee then doing some linear
algebra we have that $n_1,n_2$ in (\ref{eqn:vs1}) are given by
\bel{eqn:linalg} (n_1,n_2) = (p,q) + m(\sigma_y(w),-\sigma_x(w)); ~m
\in \Z \ee 

We now investigate the situation where $w(x,y)=u$ has rank 1 and rank
2 solutions, i.e $V(S) \supsetneq V(S_1)$. We first want to understand
$F_{R(S_1)}$.

\begin{lem} Suppose that $w(x,y)$ is not primitive nor a proper power
  and suppose moreover that $w(x,y)=u$ admits rank 1 and rank 2
  solutions. Then there $F_{R(S_1)}$ is isomorphic to
  $\brakett{F,s|[u,s]=1}=F_1$. The $F-$morphism
  $\pi_1:F_{R(S_1)}\rightarrow F_1$ given by
  \bel{eqn:non-strict-map}\pi_1(x) = u^ps^{\sigma_y(w)} = \ol{x};~
  \pi_1(y) = u^qs^{-\sigma_x(w)} = \ol{y} \ee where $p,q$ are as in
  equation (\ref{eqn:gcd}), realizes this isomorphism.
\end{lem}
\begin{proof}
  Consider the $F-$epimorphism $\pi_1:F_{R(S_1)} \rightarrow
  \brakett{F,s|[u,s]=1}=F_1$ given by (\ref{eqn:non-strict-map}) On one
  hand we see that $\pi_1$ is surjective which gives an injection
  \bel{eqn:pullback}\Hom(F_1,F) \hookrightarrow \Hom(F_{R(S_1)},F)\ee
  via pullbacks $f\mapsto f \circ \pi_1$. On the other hand $F_1$, a
  free rank 1 extension of a centralizer, is fully residually free. On
  the third hand the group $\Delta_1$ of canonical $F$ automorphisms
  of $F_1$ is generated by the automorphism given by:
\[\delta:\left\{\begin{array}{ll}s \mapsto su& \\ f \mapsto f
    & f\in F \end{array}\right.\] and if we consider the
$F-$epimorphism $\pi_2:F_1\rightarrow F$ given by $\pi_2(s)=u$ then we
immediately see that the set \[V = \{(\pi_2(\sigma^m(\ol{x})),
\pi_2(\sigma^m(\ol{y})) \in F\times F|\sigma \in \Delta_1\}\] of
images of $(x,y)$ via the mappings $\pi_2 \circ \sigma \circ \pi_1,
\sigma \in \Delta_1$ coincides with $V(S_1)$. And since $\Hom(F_1,F) =
\{\pi_2\circ \sigma| \sigma \in \Delta_1\}$ we get that the
correspondence (\ref{eqn:pullback}) is in fact a bijective
correspondence. It follows that $F_{R(S_1)} \approx_F F_1$.
\end{proof}

\begin{prop}\label{prop:rank1-solutions} Let $w(x,y)$ be non primitive
  and not a proper power. Suppose moreover that $w(x,y)=u$ has rank 1
  and rank 2 solutions. Then
  \begin{itemize}
  \item[(i)] if $\FRS$ is as in $1.$ in Corollary
    \ref{cor:allowable-jsj} , then $V(S_1)$ is represented by the
    following branch in $\Diag(\FRS,F)$:\bel{eqn:case1} \xymatrix{\FRS
      \ar[r]^{\pi_1} & F_1\ar@(ur,ul)_{\sigma} \ar[r]^{\pi_2}&F}\ee
    where $\sigma \in \Delta_1$.
  \item[(ii)] If $\FRS$ is as in $2.$ in Corollary
    \ref{cor:allowable-jsj} , then $V(S_1)$ is represented by the
    following branch in $\Diag(\FRS,F)$:
    \bel{eqn:case2}\xymatrix{\FRS\ar@(ur,ul)_{\sigma} \ar[r]^{\pi_3} &
      F}\ee where $\sigma \in \Delta$ and $\pi_3=\pi_2\circ\pi_1$
\end{itemize}
Where $\pi_1, \pi_2$ and $\Delta_1$ were defined in the previous
proof.
\end{prop}
\begin{proof}
  We first note that if $\FRS$ corresponds to case $3.$ of Corollary
  \ref{cor:allowable-jsj}, then the equality (\ref{eqn:gcd}) is
  impossible. In both possible cases we have epimorphisms
  \bel{eqn:chain}\xymatrix{\FRS \ar[r]^{\pi_1} & F_1
    \ar[r]^{\pi_2}&F}\ee We saw that all solutions rank 1 solutions
  factor through $\pi_1$. If $\FRS$ is as in $1.$ in Corollary
  \ref{cor:allowable-jsj} then $\Delta$ is generated by $\gamma_w$,
  now since $\pi_1\circ\gamma_w = \pi_1$ we have that solutions in
  $V(S_1)$ must factor through $F_1$ and are parametrized by
  $\Delta_1$.

  If $\FRS$ is as in $2.$ in Corollary \ref{cor:allowable-jsj}, then
  $\bxyk$ splits as \[\brakett{H,t|t^\mo p t = q}; p,q \in H\], moreover
  by Lemma \ref{lem:allowable-splittings} we have that
  $\bxyk=\brakett{p,t}$. We consider this basis of $\bxyk$. Let
  $\pi_1(t)=\ol{t},\pi_1(p)=\ol{p}$, then the subgroup
  $\Z\oplus\Z\approx A = \brakett{u,s} \leq F_1$ is generated by
  $\ol{p},\ol{t}$. We note that in $\FRS$, as written as a word in
  $\{p,t\}^\pmo, w(x,y)=w'(p,t)=u$ has exponent sum zero in the letter
  $t$.  Since $A$ is the abelianization of $\bxyk$, we have that in
  $A, u=0\ol{t}+n\ol{p}$ and since $u$ lies in a minimal generating
  set of $A$ we must have $n=\pm 1$. It therefore follows that for the
  Dehn twist $\tau$, which sends $t\mapsto tq$, we have
  $\pi_1\circ\tau=\delta \circ \pi_1$, where $\delta$ is the generator
  of $\Delta_1$. It follows that the canonical $F$-automorphisms of
  $F_1$ in (\ref{eqn:chain}) can be ``lifted'' to $\FRS$ and the
  branch (\ref{eqn:case2}) gives us a parametrization of $V(S_1)$.
\end{proof}

\subsection{Solutions of rank 2} 

Before being able to make our finiteness arguments we need some
preliminary setup. We will study more closely mappings
$F(x,y)\rightarrow F$.
\begin{defn}\begin{itemize}\item[(i)]  Let $(f_1,f_2)$ be a pair of words
    in a free group, then an \emph{elementary Nielsen move} (e.N.m.) is
    a mapping of the form\[ (f_1,f_2) \mapsto
    (f_1,(f_2^{\epsilon_1}f_1^{\epsilon_2})^{\epsilon_3}) \tr{~or~}
    (f_1,f_2)\mapsto
    ((f_1^{\epsilon_1}f_2^{\epsilon_2})^{\epsilon_3}),f_2)\] with
    $\epsilon_1,\epsilon_3 \in \{-1,1\}$ and $\epsilon_2 \in
    \{-1,0,1\}$.
  \item[(ii)] For $F(x,y)$, the free group on the basis $\{x,y\}$, an
    \emph{ elementary Nielsen transformation} (e.N.t.) is an element
    of $Aut(F(x,y))$ that is defined by the mappings:\[
    \left\{\begin{array}{l}
        x \mapsto (x^{\epsilon_1}y^{\epsilon_2})^{\epsilon_3}\\
        y \mapsto y \end{array}\right. \tr{or~}
    \left\{\begin{array}{l}
        x \mapsto x\\
        y \mapsto (y^{\epsilon_1}x^{\epsilon_2})^{\epsilon_3}
      \end{array}\right.\] with $\epsilon_1,\epsilon_3 \in \{-1,1\}$
    and $\epsilon_2 \in \{-1,0,1\}$.
  \end{itemize}
\end{defn}

\begin{lem}\label{lem:move-trans-corresp} Suppose $\phi$, given by
  $(x_0,y_0) \in F\times F$, is a rank 2 solution of $w(x,y)=u$, let
  \[\xymatrix{(x_0,y_0)\ar[r]_{m_1} & \ldots \ar[r]_{m_n} &
    (x_n,y_n)}\] be a sequence of e.N.m. then\begin{itemize} 
  \item[(i)] there is a corresponding sequence of e.N.t
    $t_1,\ldots,t_n$ such that letting $w_0(x,y)=w(x,y)$ and
    $w_{j+1}(x,y)=t_{j+1}(w_j(x,y))$ we have the equalities
    \bel{eqn:invariants}u=w_0(x_0,y_0)= \ldots = w_n(x_n,y_n) \ee
  \item[(ii)] Let \bel{eqn:associated-auto}\alpha=t_n \circ \ldots
    \circ t_1 \in Aut(F(x,y))\ee then the mapping
    $\phi'=\phi\circ\alpha^\mo:F(x,y)\rightarrow F$ is given by the pair $(x_n,y_n)$
\end{itemize}
\end{lem}

\begin{proof}[sketch of proof] Noting that a rank 2 solution
  isomorphically identifies the subgroup $\bxyk \leq \FRS$ with a rank
  2 subgroup of a free group, the proof is essentially the same as the
  proof that elementary Nielsen transformations generate the
  automorphisms of a f.g. free group (See Proposition I.4.1. of
  \cite{Lyndon-Schupp-1977}).
\end{proof}

The reader can look at Section I.2 of \cite{Lyndon-Schupp-1977} for
the necessary background for the next lemma.

\begin{lem}\label{lem:canonical-pair} Fix a basis $X$ of $F$, then to
  any subgroup $H\leq F$ of rank $n$ we can canonically associate an
  ordered set of Nielsen reduced generators $(j_1,\ldots,j_n)$,
  moreover this ordered set can be obtained from any ordered $n-tuple$
  of generators $(h_1,\ldots,h_n)$ via a sequence of e.N.m.
\end{lem}

We now give names to all of these:
\begin{defn}\label{defn:all-the-names} Let $\phi$, given by
  $(x_0,y_0)$, be a solution of $w(x,y)=u$. Let
  \[\xymatrix{(x_0,y_0)\ar[r]_{m_1} & \ldots \ar[r]_{m_n} &
    (x_n,y_n)}\] be the sequence of e.N.m. that brings the pair
  $(x_0,y_0)$ to the canonical pair $(x_n,y_n)$ of generators of
  $\brakett{x_0,y_0}$ guaranteed by Lemma
  \ref{lem:canonical-pair}. Then we have:\begin{itemize}
  \item The pair $(x_n,y_n)$ is called the \emph{terminal pair} of
    $\phi$ (denoted $tp(\phi)$.)
  \item The word $w_n(x,y)\in \bxyk$ in (\ref{eqn:invariants}) is
    called the \emph{terminal word} of $\phi$ (denoted $tw(\phi)$.)
  \item The automorphism $\alpha \in Aut(F(x,y))$, is the \emph{automorphism
    associated} to $\phi$ (denoted $\alpha_\phi$.)
  \end{itemize}
\end{defn}

\begin{prop}\label{prop:finiteness} Let $S=\{w(x,y)=u\}$ and let
  $U\subset V(S)$ be the open subvariety of rank 2 solutions, then
  there are only finitely many possible terminal pairs and terminal
  words that can be associated to solutions $\phi\in U$.
\end{prop}

\begin{proof} Fix a basis $X$ of $F$, we first show finiteness of
  possible terminal pairs. 

  Let $\phi$ be a solution, given by $(x_0,y_0)$ and let
  $H=\brakett{x_0,y_0}\leq F$ and let $\Gamma$ be the Stallings graph
  for $H$ (See, for instance, \cite{Stallings-top-finite-graphs}.)
  Then there is a path in $\Gamma$ with label $u$. We also have that
  Nielsen generators can be read directly off $\Gamma$ (see
  \cite{KM-fold-2002}) as labels of simple closed paths. If we define
  the \emph{radius} of $\Gamma$ to be the distance between the
  basepoint of $\Gamma$ and the ``farthest'' vertex, then we see that
  the length of Nielsen the generators $(x_m,y_m)$ is bounded by two
  times the radius. Moreover since $w(x,y)$ is neither primitive nor a
  proper power in $F(x,y)\approx H$, $u$ is not primitive nor a proper
  power in $H$. It follows that the reduced path in $\Gamma$ labeled
  $u$ must cover the whole graph which means $|u|$ is at least twice
  the radius, hence \[ |x_m|,|y_m|\leq |u|\] so the number of possible
  terminal pairs is bounded.

  Consider now the terminal word $w_n(x,y)$. Since $(x_m,y_m)\in F
  \times F$ is a Nielsen reduced pair we have that\[
  |w_n(x,y)|_{\{x,y\}} \leq |w_n(x_n,y_n)|_X = |u|_X\] which bounds the
  number of terminal words.
\end{proof}

We now connect all these ideas to solutions of equations. The next
observation is obvious but critical.

\begin{lem}\label{lem:restriction} Let $\FRS$ be the coordinate
    group of $w(x,y)=u$, with $w(x,y)$ not primitive, not a proper
    power and such that $w(x,y)$ has a rank 2 solution. Then the group
    of $F$-automorphisms of $\FRS$ are induced by the automorphisms of
    the free subgroup $\brakett{x,y}$ that fix $w(x,y)$.
\end{lem}

\begin{prop}\label{prop:scheme} Suppose that $\phi$ and $\phi'$ are
    solutions $\FRS \rightarrow F$ of $w(x,y)=u$. And suppose moreover
  that $tp(\phi) = tp(\phi')$ and $tw(\phi) = tw(\phi')$, then there
  is an automorphism $\beta \in Aut_F(\FRS)$ such that $\phi' =
  \phi\circ\beta$.
\end{prop}

\begin{proof}Let $\phi$ be given by $(x_0,y_0)$ and let $\phi'$ be
  given by $(x_0',y_0')$. Then we have a sequence of e.N.m. \[
  \xymatrix{ (x_0,y_0)\ar[r]_{m_1} & \ldots \ar[r]_{m_n} &
    tp(\phi)=tp(\phi') & \ar[l]^{m'_r} \ldots & \ar[l]^{m'_1}
    (x_0',y_0')}\] And we have automorphisms
  $\alpha_\phi,\alpha_{\phi'}$ such that $\alpha_{\phi}(w(x,y)) =
  \alpha_{\phi'}(w(x,y)) = tw(\phi)$. On one hand we have that
  $\beta=\alpha_{\phi}^\mo\circ\alpha_{\phi'} \in \stw$, so, by Lemma
  \ref{lem:restriction}, $\beta \in Aut(F(x,y))$ extends to an
  automorphism of $\FRS$. We moreover have by Lemma
  \ref{lem:move-trans-corresp} we have that the mappings
  $F(x,y)\rightarrow F$, $\phi'\circ\alpha_{\phi'}^\mo =
  \phi\circ\alpha_{\phi}^\mo$ which means that
  \[\phi'=\phi\circ\alpha_{\phi}^\mo\circ\alpha_{\phi'} =
  \phi\circ\beta\]
\end{proof}

So we have proved that all rank 2 solutions are obtained from a finite
family $\phi_1, \ldots \phi_N$ of solutions and precomposition with
$F-$automorphism of $\FRS$. Nothing so far has been said about
canonical automorphisms. 

\begin{defn}\label{defn:canonical-aut-minimality} Let $\Delta \leq
  Aut(\FRS)$ be the group of canonical $F-$auto\-morphisms of $\FRS$
  associated to a cyclic JSJ decomposition modulo $F$. Let $\phi,\phi'
  \in \Hom(\FRS,F)$, we say $\phi\sim_\Delta \phi'$ if there is a
  $\sigma \in \Delta$ such that $\phi\circ\sigma=\phi'$. $\phi \in
  \Hom(\FRS,F)$ is \emph{minimal} if after fixing a basis $X$ of $F$
  the quantity $l_f=|\phi(x)|+|\phi(y)|$ is minimal among all
  $F$-morphisms in $\phi$'s $\sim_\Delta$ equivalence class.
\end{defn}

We wish to show that there are only finitely many
$\Delta$-\emph{minimal} rank 2 solutions to $w(x,y)=u$. In light of
Proposition \ref{prop:scheme}, this is equivalent to the statement
$[\stab(w):\Delta]<\infty$.

\subsubsection{Proving finite index}
In \cite{BKM-Isomorphism}, it is proved that for freely indecomposable
fully residually free groups, the subgroup canonical automorphism is
of finite finite index in the group of outer automorphisms.
Unfortunately, the result as formulated does not cover the case
involving only automorphisms modulo $F$. We therefore prove this fact
directly. What we will essentially show is that \emph{the internal
  F-automorphisms are of finite index in the whole group of
  F-automorphisms}. The main pillars of the argument are that the JSJ
decomposition is \emph{canonical} in the sense of $(4)$ of Theorem
\ref{thm:jsj} and the following Theorem:

\begin{thm}[Corollary 15.2 of \cite{KM-JSJ}]\label{thm:finite-index}
  Let $G$ be a nonabelian fully residually free group, and let
  $\mathcal{A} = \{A_1,\ldots,A_n\}$ be a finite set of maximal
  abelian subgroups of $G$. Denote by $Out(G;\mathcal{A})$ the set of
  those outer automorphisms of $G$ which map each $A_i \in
  \mathcal{A}$ onto a conjugate of itself. If $Out(G;\mathcal{A})$ is
  infinite, then $G$ has a nontrivial abelian splitting, where each
  subgroup in $\mathcal{A}$ is elliptic. There is an algorithm to
  decide whether $Out(G;\mathcal{A})$ is finite or infinite. If
  $Out(G;\mathcal{A})$ is infinite, the algorithm finds the splitting.
  If $Out(G;\mathcal{A})$ is finite, the algorithm finds all its
  elements.
\end{thm}

This next lemma follows immediately from the fact that in free groups
$n^{th}$ roots are unique and centralizers of elements are cyclic.

\begin{lem}\label{lem:stable-letter-fun}Let $\bxyk$ be a free group
  and suppose \[\bxyk = \brakett{H,t|t^\mo pt=q}; p,q \in H-\{1\}\]
  Suppose that for some $g\in \bxyk$ we have the equality \[g^\mo p g
  = q\] then $g=tq^j$ for some $j\in \Z$.
\end{lem}

\begin{prop}\label{prop:finite-index}$\Delta \leq Aut(F(x,y))$ is of
  finite index in $\stw$
\end{prop}

\begin{proof}
If $w$ is conjugate to either $[x,y]$ or $[y,x]$ then the result
follows immediately since the $\stw$ coincides with the
automorphisms given in Corollary
\ref{cor:allowable-jsj}. (See, for instance,
\cite{Malcev-1962}.) We first concentrate on the case where the JSJ of
$\FRS$ is as in case $2.$ of Corollary \ref{cor:allowable-jsj}.

Suppose the induced splitting of $\bxyk$ is of the form
\[\bxyk=\brakett{H,t|t^\mo p t=q} p,q \in H-\{1\}\] Let $\alpha \in \stab(w) \leq
Aut(\bxyk)$, then we can extend $\alpha$ to
$\widehat{\alpha}:\FRS\rightarrow \FRS$. We wish to understand the
action of $\widehat{\alpha}$ on $\FRS$. First note that $\ahat$
restricted to $F$ is the identity and $\ahat(\bxyk)=\bxyk$ On the
other hand, $\ahat$ gives another cyclic JSJ decomposition $D_1$
modulo $F$:\bel{eqn:neq-jsj}\FRS = F*_{u=w(x,y)} \brakett{\ahat(H),
  \ahat(t) | \ahat(t)^\mo \ahat(p) \ahat(t) = \ahat(q)}\ee with $w \in
\ahat(H)$. By Theorem \ref{thm:jsj} $(4)$, $D_1$ can be obtained from
$D$ by a sequence of slidings, conjugations and modifying boundary
monomorphisms.

$\ahat(H)\cap F = \brakett{w}$, and $H$ must be obtained from
$\ahat(H)$ as in $(4)$ of Theorem \ref{thm:jsj}, i.e. by slidings,
conjugating boundary monomorphisms and conjugations. The only inner
automorphism of $\FRS$ that fixes $w$ is conjugation by $w^k; k \in \Z$; (use
Bass-Serre theory and properties of free groups) and since $\ahat(H)$
and $H$ are attached to $F$ at $\brakett{w}$, slidings will have no
effect. It follows that $\ahat(H)=H$. Applying Theorem \ref{thm:jsj}
again forces $p,q$ to be conjugate in $H$ to $\ahat(p),\ahat(q)$
[respectively or in the other order]. We now have strong
information enough on the dynamics of $\stab(w)$ to apply Theorem
\ref{thm:finite-index}.

Indeed since $\ahat(H)=H$, we have a natural homomorphism
$\rho:\stab(w)\rightarrow \tilde{\stw}\leq Aut(H)$ given by the
restriction $\alpha \gt \alpha|_H$. Moreover we see that any almost
reduced cyclic splitting of $H$ modulo $\{\brakett{w}, \brakett{p},
\brakett{q}\}$ must be trivial, otherwise contradicting Lemma
\ref{lem:allowable-splittings}.  Let $\pi:Aut(H)\rightarrow Out(H)$ be
the canonical map (i.e. quotient out by $Inn(G)$, the subgroup of
inner automorphisms). It therefore follows from Theorem
\ref{thm:finite-index} that the image $\pi\circ\rho(\stw)=\ol{\stw}$
must be finite.

First note that $Inn(H)\cap \tilde{\stw} = \bgwk$ which means that
\[\ol{\stw}\approx \tilde{\stw}/\bgwk\] and this isomorphism is
natural. Let $\alpha \in \ker{\rho}$ then we must have that $\alpha|_H
= 1$. In particular we have \[\alpha(t)^\mo p \alpha(t)=q\] which by
Lemma \ref{lem:stable-letter-fun} implies that $\alpha(t)=tq^j$ it
follows that $\ker(\rho)\leq \brakett{\tau}$. The other inclusion is
obvious so \[\ker(\rho)=\brakett{\tau}\] There is a bijective
correspondence between subgroups $K$ of $\tilde{\stw}$ and subgroups
of $\stw$ that contain $\brakett{\tau}$ given by $K \gt \rho^\mo(K)$.
Moreover this correspondence sends normal subgroups to normal
subgroups. It follows that $ker(\pi\circ\rho)=\brakett{\tau,\gamma_w}$
and so we get:\[\stw/\brakett{\tau,\gamma_w}\approx \ol{\stw}\] which
is finite. It follows that $[\stw:\brakett{\tau,\gamma_w}]<\infty$.

In the case where $D$, the cyclic JSJ of $\FRS$ modulo $F$ is as in
case $1.$ of Corollary \ref{cor:allowable-jsj} then again elements of
$\alpha\in\stw$ will give new splittings
$\FRS=F*_{u=w(x,y)}\ahat(H)$. Arguing as before, we get that
$\ahat(H)=H$ and we can apply Theorem \ref{thm:finite-index} with
$\mathcal{A}=\{\brakett{w}\}$. We get that $Out(H;\mathcal{A})\approx
\stw/\bgwk$ must be finite, otherwise $H$ could split further,
contradicting the fact that $D$ was a JSJ splitting, and the result
follows.
\end{proof}

By Lemma \ref{lem:restriction}, Propositions \ref{prop:finiteness},
\ref{prop:scheme}, and \ref{prop:finite-index} we get the second half
of our main result:
\begin{prop}\label{prop:rank2-finite} Suppose that $w(x,y)$ is not a
  proper power, nor is it primitive. Then there are finitely many
  $\Delta-$minimal rank 2 solutions to the equation $w(x,y)=u$.
\end{prop}

\subsection{The description of
  $V(\{w(x,y)u^\mo\})$}\label{sec:description}
These next two results now follows immediately from Proposition
\ref{prop:rank2-finite}, \ref{prop:rank1-solutions}, Corollary
\ref{cor:allowable-jsj}, Lemma \ref{lem:allowable-splittings}
and Theorem \ref{thm:parametrization}.

\begin{thm}Suppose that $w(x,y)=u$ has rank 2 solutions and that
  $w(x,y)$ is not a power of a primitive element. Then the possible
  Hom diagrams are given in Figure \ref{homdiagrams}.
\end{thm}

\begin{thm}\label{thm:main-result} Suppose that $w(x,y)=u$ has rank 2 solutions and that
  $w(x,y)$ is neither primitive nor a proper power. Let $\{\phi_i|
  i\in I\}$ be the collection of $\Delta-$minimal solution. Then
  $V(S)=V(S_1)\cup V'$, where $V'=V(S)-V(S_1)$,is given by the
  following:\begin{enumerate}
  \item $\FRS\approx F*_{u=w(x,y)}\brakett{x,y}$, let $\phi_i(x)=x_i,
  \phi_i(y)=y_i$ then $V(S)=V(S_1)\cup V'$ where \[V'=\{(u^{-n}x_iu^n,
  u^{-n}y_iu^n) | i \in I \tr{~and~} n \in \Z\}\] and if the exponent
  sums $\sigma_x(w),\sigma_y(w)$ of $x,y$ respectively in $w$ are
  relatively prime, then $V(S_1)$ is non empty and is given by
  (\ref{eqn:vs1}).
\item $\FRS\approx F*_{u=w(x,y)}\brakett{H,t|t^\mo p t = q}$,
  $H=\brakett{p,q}$ and we can write $x,y \in \bxyk$ as words
  $x=X(p,q,t), y=Y(p,q,t)$. Let
  $\phi_i(p)=p_i,\phi_i(q)=q_i,\phi_i(t)=t_i$ then we have that
  $V(S)=V(S_1)\cup V'$ where \begin{gather} \notag
\begin{split}
 V'=&\{(X(u^{-n}p_i u^n,u^{-n}q_i
u^n,u^{-n}t_iq_{i}^{m}u^n),\\ &Y(u^{-n}p_i u^n,u^{-n}q_i
u^n,u^{-n}t_iq_{i}^{m}u^n)) \, \mid\, i\in I, n,m \in \Z\}
\end{split}
\end{gather} and if the exponent sums $\sigma_x(w),\sigma_y(w)$ of
    $x,y$ respectively in $w$ are relatively prime, then $V(S_1)$ is
  non empty and is given by (\ref{eqn:vs1}).
\item $\FRS \approx F*_{u=w(x,y)}Q$ where $Q$ is a QH subgroup and, up
  to rational equivalence, $Q=\brakett{x,y,w|[x,y]w^\mo}$. Then
  $V(S_1)$ is empty. Let $\phi_i(x)=x_i, \phi_i(y)=y_i$ then
  \[V(S)=\{(X_\sigma(x_i,y_i),Y_\sigma(x_i,y_i))|\sigma \in \Delta\}\]
  where the words $\sigma(x)=X_\sigma(x,y), \sigma(y)=Y_\sigma(x,y)
  \in\bxyk$.
\end{enumerate}
\end{thm}

We finally note that unless $w(x,y)=u$ is orientable quadratic, then
solutions are given by ``one level parametric'' words (see
\cite{Lyndon-1960} for a definition.)

\section{An Interesting Example}
The Hom diagrams given for $w(x,y)=u$ were very simple. In particular,
modulo the slight technicalities of Theorem \ref{thm:main-result} item
1, we can say that; unless $w(x,y)$ is a power of a primitive element;
there are only finitely many minimal solutions to $w(x,y)=u$ with
respect to a group of canonical automorphisms. This translates as the
Hom diagram having only one ``level''. This also means that all
\emph{fundamental sequences} or \emph{strict resolutions} of $\FRS$
have length 1 (see \cite{KM-IrredII} or \cite{Sela-DiophI},
respectively for definitions.) It is natural to ask this holds true
for general equations in two variables. We answer this negatively:

\begin{thm}\label{thm:two-levels} Let $F=F(a,b)$ then the Hom diagram
  associated to the equation with variables $x,y$
  \bel{eqn:megaquation} [a^\mo b a[b,a][x,y]^2x,a]=1 \ee has branches
  corresponding to rank 2 solutions that have length at least $2$.
\end{thm}
\begin{proof} First note that via Tietze transformations, we have the
  following isomorphism:\begin{eqnarray*}
    &&\brakett{F,x,y|[a^\mo b a[b,a][x,y]^2x,a] =1}\\
    &&\approx \brakett{F,x,y,t | [x,y]^2x= [a,b]a^\mo b^\mo a t;
      [t,a]=1}\end{eqnarray*} 

   Let $w(x,y)=[x,y]^2x$ and let $u=[a,b]a^\mo b^\mo a t$. We now embed
  $G=\brakett{F,x,y,t | w(x,y)=u, [t,a]=1}$ into a chain of extensions
  of centralizers.  Let $F_1=\brakett{F,t|[t,a]=1}$ and let
  $F_2=\brakett{F_1,s| [u,s]=1}$. Let $\ol{x}=b^\mo t$ and
  $\ol{y}=b^\mo a b$. First note that\[[\ol{x},\ol{y}]^2\ol{x}=((t^\mo
  b)(b^\mo a^\mo b)(b^\mo t)(b^\mo a b))^2(b^\mo t)= [a,b]a^\mo b^\mo
  a t=u\] We now form a double, i.e. we set $x=\ol{x}^s, y = \ol{y}^s$
  and let $H =\bxyk = \brakett{\ol{x},\ol{y}}^s$. By Britton's Lemma
  we have that $H \cap \F_1 = \brakett{u}$ and it follows that
  $\brakett{F,x,y}$ is isomorphic to the amalgam $F_1*_\brakett{u} H = G$.
  Since chains of extensions of centralizers of $F$ are fully
  residually $F$. We have that our equation (\ref{eqn:megaquation}) is
  an irreducible system of equations, we write $\FRS=G$. We note that
  we have the nontrivial cyclic splitting \[ D:~\FRS \approx
  F_1*_\brakett{u=w(x,y)}\bxyk\] moreover since $w(x,y)=[x,y]^2x$
  cannot belong to a basis (see \cite{CMZ-Fab-basis}) of $\bxyk$ we
  have that $\FRS$ is freely indecomposable modulo $F_1$. On the other
  hand, if we take the Grushko decomposition of $\FRS$ modulo
  $F$\[\FRS = \F*K_1*\ldots K_n; ~F \leq \F\] we see that we must have
  $F_1\leq \F$ since $[t,a]=1 \Rightarrow t \in \F$. It follows that
  $\FRS$ is actually freely indecomposable modulo $F$. It follows that
  $D$ can be refined to a cyclic JSJ decomposition modulo $F$.

  Suppose towards a contradiction that all branches of the Hom diagram
  for $\Hom(\FRS,F)$ corresponding to rank 2 solutions had length 1.
  This means that there are finitely many minimal rank 2 solutions
  $\phi:\FRS \rightarrow F$. On one hand the element $t$ must be sent
  to arbitrarily high powers of $a$, since $\FRS$ is fully residually
  $F$. On the other hand, for there to be a canonical automorphism of
  $\FRS$ that sends $t \mapsto ta^n$, there must be a splitting $D'$
  of $\FRS$ with some conjugate of $\brakett{a}$ as a boundary
  subgroup, but $u$ would have to be hyperbolic in such a splitting,
  and since $\brakett{a}$ is elliptic in $D$, we would have an
  elliptic-hyperbolic splitting which by Theorem \ref{thm:hyp-hyp}
  would contradict free indecomposability modulo $F$.
\end{proof}

We now provide some illustration. We determined that $\FRS =
F_1*_\brakett{u=w(x,y)}\bxyk$ with $u=[a,b]a b^\mo a^\mo t$. Now the
mapping $x\mapsto x^u$ and $y\mapsto y^u$ extends to a canonical
automorphism of $\FRS$ and along some branch there must be another
canonical automorphism that maps $t \mapsto ta^r$. By checking
directly we see that $\phi:\FRS\mapsto F$ given by $x=b^\mo a, y=b^\mo
a b$ is a solution, so we can get the family of solutions:
\[x= ([a,b]a b^\mo a^\mo a^n)^m(b^\mo a)([a,b]a b^\mo a^\mo
a^n)^{-m}\] \[y=([a,b]a b^\mo a^\mo a^n)^m(b^\mo a b)([a,b]a b^\mo
a^\mo a^n)^{-m}\] with $n,m$ in $\Z$. Notice that no precomposition by
a canonical automorphism of $\FRS$ can affect the $n$ parameter. It
follows that the set of solution of (\ref{eqn:megaquation}) can not be
given by precomposing a finite collection of maps
$\phi_1,\ldots,\phi_n: \FRS \rightarrow F$ with canonical
automorphisms.

\def\cprime{$'$}

\end{document}